\documentclass[11pt,reqno]{article}
\usepackage{latexsym, amsmath, amssymb, amsthm, a4, epsfig}
\usepackage{graphicx}
\usepackage{color}

\newtheorem{theorem}{Theorem}[section]

\newtheorem{lemma}[theorem]{Lemma}
\newtheorem{prop}[theorem]{Proposition}

\newcommand{\ds}{\displaystyle}
\newcommand{\p}{\partial}

\newcommand{\eqnref}[1]{(\ref {#1})}

\newcommand{\Rbb}{\mathbb{R}}

\newcommand{\la}{\langle}
\newcommand{\ra}{\rangle}

\newcommand{\Fcal}{\mathcal{F}}

\newcommand{\Hcal}{\mathcal{H}}
\newcommand{\Lcal}{\mathcal{L}}
\newcommand{\Kcal}{\mathcal{K}}

\newcommand{\Ga}{\alpha}

\newcommand{\Gd}{\delta}
\newcommand{\Ge}{\epsilon}

\newcommand{\Gf}{\phi}
\newcommand{\Gvf}{\varphi}
\newcommand{\Gg}{\gamma}
\newcommand{\Gc}{\chi}

\newcommand{\Gl}{\lambda}
\newcommand{\Gn}{\eta}

\newcommand{\Gt}{\theta}

\newcommand{\Gr}{\rho}

\newcommand{\Gs}{\sigma}

\newcommand{\Go}{\omega}

\newcommand{\Gz}{\zeta}

\newcommand{\GF}{\Phi}
\newcommand{\GG}{\Gamma}

\newcommand{\GP}{\Pi}

\newcommand{\GS}{\Sigma}

\newcommand{\GO}{\Omega}

\newcommand{\beq}{\begin{equation}}
\newcommand{\eeq}{\end{equation}}

\def\ol{\overline}

\numberwithin{equation}{section}
\numberwithin{figure}{section}

\begin{document}

\title{Spectral structure of the Neumann--Poincar\'e operator on thin ellipsoids and flat domains\thanks{\footnotesize This work was supported by NRF (of S. Korea) grants 2019R1A2B5B01069967, 2021R1A2B5B02-001786 and by JSPS of Japan KAKENHI grants JP19K14533, 20K03655 and 21K13805.}}

\author{Kazunori Ando\thanks{Department of Electrical and Electronic Engineering and Computer Science, Ehime University, Ehime 790-8577, Japan. Email: {\tt ando@cs.ehime-u.ac.jp}.}
\and Hyeonbae Kang\thanks{Department of Mathematics and Institute of Applied Mathematics, Inha University, Incheon 22212, S. Korea. Email: {\tt hbkang@inha.ac.kr}.}
\and Sanghyuk Lee\thanks{Department of Mathematical Sciences, Seoul National University, Seoul 08826, S. Korea. Email: {\tt shklee@snu.ac.kr} }
\and Yoshihisa Miyanishi\thanks{Department of Mathematical Sciences, Faculty of Science, Shinshu University, A519, Asahi 3-1-1, Matsumoto 390-8621, Japan. Email: {\tt miyanishi@shinshu-u.ac.jp}}.}
\date{}
\maketitle

\begin{abstract}
We investigate the spectral structure of the Neumann--Poincar\'e operator on thin ellipsoids.
Two types of thin ellipsoids are considered:
long prolate ellipsoids and flat oblate ellipsoids.
We show that the totality of eigenvalues of the Neumann--Poincar\'e operators on a sequence of the prolate spheroids is densely distributed in the interval $[0,1/2]$ as their eccentricities tend to $1$, namely, as they become longer. We then
prove that eigenvalues of the Neumann--Poincar\'e operators on the oblate ellipsoids are densely distributed in the interval $[-1/2,1/2]$ as the ellipsoids become flatter. In particular, this shows that even if there are at most finitely many negative eigenvalues on the oblate ellipsoids, more and more negative eigenvalues appear as the ellipsoids become flatter.
We also show a similar spectral property for flat three dimensional domains.
\end{abstract}

\noindent{\footnotesize {\bf AMS subject classifications}. 35J05 (primary), 35P05 (secondary)}

\noindent{\footnotesize {\bf Key words}. Neumann--Poincar\'e operator, spectrum, prolate spheroids, oblate ellipsoids, negative eigenvalue}

\section{Introduction}

For a bounded domain $\GO$ with the Lipschitz continuous boundary in $\Rbb^d$, $d=2,3$, the Neumann--Poincar\'e (abbreviated by NP) operator associated with $\p\GO$ is the boundary integral operator on $\p\GO$ defined by
\beq
\Kcal_{\p\GO} [\Gvf](x) = \frac{1}{\Go_d} \int_{\p\GO} \frac{\la y-x, {\nu(y)} \ra}{|x-y|^d} \Gvf(y) dS(y), \quad x \in \p\GO,
\eeq
where $\Go_d=1/2\pi$ if $d=2$ and $1/4\pi$ if $d=3$, and $\nu(y)$ denotes the outward unit normal to $\p\GO$ at $y \in \p\GO$. It naturally appears when solving the classical Dirichlet problem using layer potentials, and is commonly called the double layer potential. The NP operator can be realized as a self-adjoint operator on $H^{1/2}(\p\GO)$, the Sobolev space on $\p\GO$ \cite{KPS} (see also the recent survey \cite{AKMP}). If $\p\GO$ is smooth ($C^{1, \Ga}$ for some $\Ga>0$ to be precise), then it is a compact operator and has a countable number of eigenvalues accumulating to $0$. It is known that the NP eigenvalues (eigenvalues of the NP operator) are confined in the interval $(-1/2, 1/2]$ (see, for example, \cite[Chapter XI, Section 11]{Kellog-book} or \cite{AKMP}).

The NP spectrum depends heavily on geometry of the surface (or the curve) on which the operator is defined. In particular, as the boundary $\p\GO$ becomes `singular' in some sense, the spectrum seems to approach to the bounds $\pm 1/2$. For example, if $\GO$ consists of two strictly convex planar domains and boundaries get closer, then more and more eigenvalues of the corresponding NP operator approach $\pm 1/2$ \cite{BT, BT2}. If a planar curvilinear domain has corners, then the essential spectrum of the NP operator is an interval whose end-points are determined by the smallest angle of the corners \cite{PP1, PP2} (see also \cite{BZ}) (in this case, essential spectrum possibly except $0$ consists of absolutely continuous spectrum \cite{KLY, Perfekt}). If a corner gets sharper and the domain becomes needle-like around the corner, then the essential spectrum approaches $[-1/2,1/2]$. For rectangles whose corner angles are $\pi/2$, the essential spectrum is fixed to be $[-1/4,1/4]$. However, it is shown in \cite{HKL} by numerical computations that there appear more and more eigenvalues outside the interval, which approaches to $\pm 1/2$, as the aspect ratio of the rectangle becomes larger.

Motivated by above observation, it is proved in the recent paper \cite{AKM21} that if $D_R$ is a rectangular shape planar domain of the aspect ratio $R$, then for any sequence $R_j$ of positive numbers tending to $\infty$ as $j \to \infty$, the NP spectra are densely distributed in $[-1/2, 1/2]$. More precisely,
\beq\label{2D}
\ol{\cup_{j=1}^\infty \Gs(\Kcal_{\p D_j})} = [-1/2, 1/2].
\eeq
Here and afterwards, $\Gs(\Kcal_{\p \GO})$ denotes the spectrum of the NP operator $\Kcal_{\p \GO}$ on $H^{1/2}(\p \GO)$. This proves that more and more NP eigenvalues appear outside the  essential spectrum $[-1/4,1/4]$ to densely fill up $[-1/2, 1/2] \setminus [-1/4,1/4]$. This is in accordance with the numerical finding in \cite{HKL}. A similar spectral property is shared by a sequence of ellipses. Since the NP eigenvalues on ellipses are explicitly known, it can be shown without difficulty that the spectral property \eqnref{2D} holds for ellipses of the form $x_1^2/R_j^2 + x_2^2 < 1$ (see \cite{AKM21} for a proof). It says that even if the totality of the spectrum is countable, it is dense in $[-1/2, 1/2]$, and it holds regardless of choices of the sequence $R_j$.

The purpose of this paper is to investigate the spectral structure of the NP operator on thin domains including ellipsoids and prove results similar to \eqnref{2D}. Unlike the two-dimensional case, there are two different kinds of thinness in three dimensions: thin and long (like prolate spheroids), thin and flat (like oblate ellipsoids). As we will see,
three-dimensional bounded domains exhibit the
NP spectral structure different from that of two-dimensional ones. In two dimensions, the NP spectrum always appears in pairs $\pm \Gl$, which is due to existence of harmonic conjugates. However, there are domains in three dimensions where the NP operators have only positive eigenvalues: the NP eigenvalues on a sphere are $1/(4n+2)$ for $n = 0, 1, 2\, \ldots$ \cite{Poi1}, and they are all positive on prolate spheroids \cite{AA}. Thus, the property \eqnref{2D} does not hold for prolate spheroids. It is shown in \cite{Ahner} that there is an oblate ellipsoid having a negative eigenvalue. To the best of our knowledge, this is the first example of three-dimensional domains with a negative NP eigenvalue.
Recently, it is proved in \cite{MR-SPMJ-20} that the NP operator on the boundary of strictly convex domains in three dimensions can have at most finitely many eigenvalues.  If the boundary of the domain has a concave part like tori, then there are (infinitely) many negative eigenvalues (see \cite{AJKKM, JK, MR-SPMJ-20}).

Here we discuss a possible advantage of having negative eigenvalues. Suppose that a three-dimensional domain $\GO$ has $k$ as its dielectric constant, while the background matrix $\Rbb^3 \setminus \GO$ does $1$. Then plasmon resonance in the quasi-static limit occurs if
\beq\label{fredholm}
\frac{k+1}{2(k-1)}= \Gl,
\eeq
where $\Gl$ is an eigenvalue of the NP operator on $\p\GO$ (see \cite{Grieser}). Since $\Gl$ lies in $(-1/2, 1/2]$, \eqnref{fredholm} can be fulfilled only when $k$ is negative (so that $\GO$ is a meta-material with the negative dielectric constant). The relation \eqnref{fredholm} can be achieved by a larger $k$ (the smaller $|k|$) if $\Gl$ is negative (see Figure \ref{negative}). This may yield an advantage in practice even though verifying it is out of reach of mathematical research. We also mention a recent work \cite{AKMN} where it is shown by numerical computation that the spectral property of the NP operator (in relation to the cloaking by anomalous localized resonance) on the torus is quite different from that on strictly convex surfaces.

\begin{figure}[ht!]
\begin{center}
\epsfig{figure=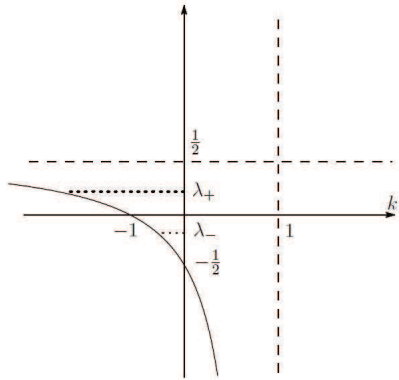,width=6cm}
\end{center}
\caption{The curve is the graph of $\frac{k+1}{2(k-1)}$ and $\Gl_+, \ \Gl_-$ are positive and negative NP eigenvalues, respectively. The relation \eqnref{fredholm} for $\Gl_-$ can be satisfied by a larger $k$ (the smaller $|k|$) than that for $\Gl_+$.} \label{negative} \end{figure}

We now present the main results of this paper. Let us begin with the prolate spheroids. Let $\GP_R$ be a prolate spheroid, namely, for $R \ge 1$,
\beq\label{prolate}
\GP_R:= \Big\{ (x_1,x_2, x_3): x_1^2 + x_2^2 + \frac{x_3^2}{R^2} < 1 \Big\}.
\eeq
If we dilate $\GP_R$ by $R^{-1}$, $\GP_R$ becomes thin. That is why we call them `thin' domains. The NP spectrum is invariant under dilation.

We obtain the following proposition for prolate spheroids.

\begin{prop}\label{prop:prolate}
Let $\GP_R$ be the prolate spheroid defined by \eqnref{prolate}. If $R_j$ is a sequence of numbers such that $R_j \ge 1$ for all $j$ and $R_j \to \infty$ as $j \to \infty$, then
\beq\label{main:prolate1}
[0, 1/2] \subset \ol{\cup_{j=1}^\infty \Gs(\Kcal_{\p \GP_{R_j}})} .
\eeq
\end{prop}

Since $\Gs(\Kcal_{\p \GP_{R}}) \subset (0, 1/2]$ if $R \ge 1$ as proved in \cite{AA}, we obtain the following theorem as an immediate consequence.

\begin{theorem}\label{thm:prolate}
Let $\GP_R$ be the prolate spheroid defined by \eqnref{prolate}. If $R_j$ is a sequence of numbers such that $R_j \ge 1$ for all $j$ and $R_j \to \infty$ as $j \to \infty$, then
\beq\label{main:prolate}
\ol{\cup_{j=1}^\infty \Gs(\Kcal_{\p \GP_{R_j}})} = [0, 1/2].
\eeq
\end{theorem}

Theorem \ref{thm:prolate} shows that totality of eigenvalues of $\Kcal_{\p \GP_{R_j}}$ is dense in $[0, 1/2]$ regardless of choice of the sequence $R_j$ as long as $R_j \to \infty$.
There are significant works on the NP spectrum on ellipsoids \cite{Ahner, AA, ADR, Mart, Ritt}. For example, NP eigenvalues on prolate spheroids are expressed in terms of values of Legendre functions (see \eqnref{ev:prolate}). However, it is unlikely that Theorem \ref{thm:prolate} (and Theorem \ref{thm:oblate} below) can be proved using those results since we do not have enough knowledge about value distributions of Legendre functions. Nonetheless, we are able to prove the following theorem based on those results, which is in good comparison with Theorem \ref{thm:prolate}: It shows that the totality (in continuum) of the NP eigenvalues on prolate spheroids covers the interval $(0,1/2]$ while Theorem \ref{thm:prolate} shows that the NP eigenvalues on a sequence of prolate spheroids, which is countable, are dense in $[0,1/2]$ regardless of the choice of the sequence.

\begin{theorem}\label{thm:prolate2}
Let $\GP_R$ be the prolate spheroid defined by \eqnref{prolate}. It holds that for any $R_0 \ge 1$,
\beq\label{prolate2}
\cup_{R\geq R_0}\Gs(\Kcal_{\p\GP_R}) = (0, 1/2].
\eeq
\end{theorem}

We then turn our attention to oblate ellipsoids. Let $a_j$ ($j=1,2$) be positive numbers.  For a positive number $R$,  let $\GO_R$ be an oblate ellipsoid defined by
\beq\label{oblate}
\GO_R:= \left\{ (x_1, x_2, x_3): \frac{x_1^2}{(a_1 R)^2} + \frac{x_2^2}{(a_2 R)^2} + x_3^2 < 1 \right\}.
\eeq
If $a_1=a_2$, then $\GO_R$ is an oblate spheroid.

We obtain the following theorem for oblate ellipsoids:
\begin{theorem}\label{thm:oblate}
Let $\GO_R$ be the oblate ellipsoid defined by \eqnref{oblate}. If $R_j$ is a sequence of positive numbers such that $R_j \to \infty$ as $j \to \infty$, then
\beq\label{main:oblate}
\ol{\cup_{j=1}^\infty \Gs(\Kcal_{\p \GO_{R_j}})} = [-1/2, 1/2].
\eeq
\end{theorem}

Theorem \ref{thm:oblate} shows that totality of eigenvalues of $\Kcal_{\p \GO_{R_j}}$ is dense in $[-1/2, 1/2]$. This is rather surprising. As mentioned earlier, $\Kcal_{\p \GO_{R_j}}$ admits at most finitely many negative eigenvalues since $\GO_R$ is strictly convex. However, \eqnref{main:oblate} says that negative eigenvalues in $\cup_{j=1}^\infty \Gs(\Kcal_{\p \GO_{R_j}})$ are dense in $[-1/2,0]$.

Proposition \ref{prop:prolate} and Theorem \ref{thm:oblate} are proved by investigating the limiting behaviour of the NP operators as $R \to \infty$. We show that the NP operator on the prolate spheroids converges (on some test functions) to a certain one-dimensional convolution operator as $R \to \infty$ (see \eqnref{L0}). We prove that the Fourier transform of the convolution kernel has values in $(0,1/2]$ and hence the operator has continuous spectrum $(0,1/2]$, and use this fact to prove Proposition \ref{prop:prolate}.  The NP operator on oblate ellipsoids converges to the two-dimensional Poisson integral evaluated at $2$ or $-2$ (this is so because oblate ellipsoids have the upper and lower parts) (see \eqnref{4100}). The Poisson integral operator has continuous spectrum $(0,1/2]$. But, since this operator is evaluated at $\pm 2$, we are able to prove Theorem \ref{thm:oblate}.

The property \eqnref{main:oblate} seems to be a generic property of thin, flat domains. To demonstrate it, we consider typical thin, flat domains other than oblate ellipsoids. To define such a domain, let $U$ be a bounded planar domain with the Lipschitz continuous boundary $\p U$. Let $\GF$ be the domain in $\Rbb^3$ whose boundary consists of three pieces, namely,
\beq\label{Fboundary}
\p \GF = \GS^+ \cup \GS^- \cup \GS^s
\eeq
where the top and bottom are given by $\GS^{\pm}= U \times \{ \pm 1\}$ and $\GS^s$ is a surface connecting $\p U \times \{ +1\}$ and $\p U \times \{ -1\}$. We assume that $\p \GF$ is Lipschitz continuous. For $R>0$ let
\beq\label{thinflat}
\GF_R:= \{(Rx_1,Rx_2, x_3): (x_1,x_2, x_3)\in \GF \}.
\eeq

We obtain the following theorem using the method of proving Theorem \ref{thm:oblate}.
\begin{theorem}\label{thm:thinflat}
Let $\GF_R$ be the domain defined by \eqnref{thinflat}. If $R_j$ is a sequence such that $R_j \to \infty$ as $j \to \infty$, then
\beq\label{main:thinflat}
\ol{\cup_{j=1}^\infty \Gs(\Kcal_{\p\GF_{R_j}})} = [-1/2, 1/2].
\eeq
\end{theorem}

One may naturally ask a question if Theorem \ref{thm:prolate} holds for cylinder-like domains or even prolate ellipsoids. One can show that \eqnref{main:prolate1} holds for such domains. But we do not know if the reverse inclusion is true. In the oblate case, the reverse inclusion is always true, namely, the NP spectrum is contained in $[-1/2, 1/2]$.

The rest of the paper is devoted to proving main results:
Proposition \ref{prop:prolate} and Theorem \ref{thm:prolate2} in Section \ref{sec:2}; Theorem \ref{thm:oblate} in Section \ref{sec:3}; Theorem \ref{thm:thinflat} in Section \ref{sec:4}.

We use standard notation of $A \lesssim B$ which means that there is a constant $C$ independent of the parameter $R$ of the given ellipsoids. The meaning of $A \gtrsim B$ is analogous, and $A \sim B$ means both $A \lesssim B$ and $A \gtrsim B$ hold.

\section{Proof of Proposition \ref{prop:prolate} and Theorem \ref{thm:prolate2}}\label{sec:2}

In this section we prove Proposition \ref{prop:prolate} (and Theorem \ref{thm:prolate} as its consequence) and Theorem \ref{thm:prolate2}. Since $\p\GP_R$ is smooth, a non-zero eigenvalue of $\Kcal_{\p \GP_{R}}$ on $L^2(\p\GP_R)$ is automatically an eigenvalue on $H^{1/2}(\p\GP_R)$. Thus it is enough to prove \eqnref{main:prolate1} assuming $\Gs(\Kcal_{\p \GP_{R}})$ is on $L^2(\p\GP_R)$.

\subsection{Parametrization of the NP operator on prolate spheroids}

Let
$$
\Gn(t)= \Gn_R(t):= \sqrt{1-t^2/R^2}.
$$
We parametrize the prolate spheroid $\GP_R$ given by \eqnref{prolate} by $x= (\Gn(x_3) \cos\Gt, \Gn(x_3) \sin \Gt, x_3)$. Let $y = (\Gn(y_3) \cos\Gf, \Gn(y_3) \sin \Gf, y_3)$. Then, straight-forward calculations yield that
\beq\label{dSpro}
\nu(y) = \left(1-\frac{y_3^2}{R^2} + \frac{y_3^2}{R^4} \right)^{-1/2} \left(\Gn(y_3) \cos\Gf, \Gn(y_3) \sin \Gf, \frac{y_3}{R^2} \right)
\eeq
and
\beq\label{dSpro2}
dS(y) = \left(1-\frac{y_3^2}{R^2} + \frac{y_3^2}{R^4} \right)^{1/2} d\Gf dy_3.
\eeq
Thus we have
\begin{align*}
&\frac{1}{4\pi} \frac{\la y-x, {\nu(y)} \ra}{|x-y|^3}  dS(y) \\
&= \frac{1}{4\pi} \frac{(\Gn(y_3)^2- \Gn(x_3)\Gn(y_3) \cos(\Gt-\Gf) + \frac{y_3}{R^2} (y_3-x_3))}{\,\,[\big( \Gn(x_3)^2+ \Gn(y_3)^2- 2\Gn(x_3)\Gn(y_3) \cos(\Gt-\Gf) \big) + (x_3-y_3)^2]^{3/2}\,}\, d\Gf dy_3.
\end{align*}

Let $g(x_3)$ be a function supported in $(-R, R)$. Define $\psi$ on $\p\GP_R$ by
\beq\label{extpro}
\psi(x)= \psi (\Gn(x_3) \cos\Gt, \Gn(x_3) \sin \Gt, x_3)= g(x_3).
\eeq
Thanks to \eqnref{dSpro2}, we have
\beq
\| \psi \|_{L^2(\p \GP_R)} \lesssim \| g \|_2.
\eeq
Additionally, if  $g(x_3)$ is supported in $(-R^{1-\Gs}, R^{1-\Gs})$ for some $\Gs \in (0, 1)$, then
we have
\beq\label{psig}
\| \psi \|_{L^2(\p \GP_R)} \sim \| g \|_2.
\eeq
Moreover, $\Kcal_{\p \GP_R}[\psi]$ can be expressed as
\beq\label{KcalHcal}
\Kcal_{\p \GP_R}[\psi](x) = \Hcal_R [g](x_3) ,
\eeq
where $\Hcal_R$ is the integral operator defined by the integral kernel $H_R(x_3, y_3)$ given by
\beq\label{H-kernel}
H_R(x_3,y_3) = \frac{1}{2\pi} \int_{0}^{\pi} \frac{ 1-R^{-2}x_3y_3- \Gn(x_3)\Gn(y_3) \cos\Gt}{ [\big( \Gn(x_3)^2+ \Gn(y_3)^2- 2\Gn(x_3)\Gn(y_3) \cos\Gt \big) + (x_3-y_3)^2]^{3/2} }\, d\Gt.
\eeq

If $x_3$ and $y_3$ lie in $(-R^{1-\Gs}, R^{1-\Gs})$, then $\Gn(x_3)$ and $\Gn(y_3)$ tend to $1$ as $R \to \infty$. Thus, formally speaking, $H_R(x_3,y_3)$ tends to $L_0(x_3-y_3)$, where
\beq\label{L0}
L_0(t) := \frac{1}{2\pi} \int_{0}^{\pi} \frac{1- \cos\Gt}{ [\big( 2- 2\cos\Gt \big) + t^2]^{3/2} } d\Gt .
\eeq
Let $\widehat{f}$ denote the Fourier transform on $\Rbb^d$, namely,
\beq\label{Fourdef}
\widehat{f}(\xi)=\Fcal[f](\xi):= \int_{\Rbb^d} e^{-2\pi i \xi \cdot x} f(x) dx.
\eeq
Note that
\beq
\widehat{L_0}(\xi) = \frac{1}{4\pi} \int_{0}^{\pi} \widehat{k}(\sqrt{2(1- \cos\Gt)} \xi) d\Gt,
\eeq
where
\beq\label{functionk}
k(t):= \frac{1}{(1+t^2)^{3/2}}.
\eeq

\begin{lemma}\label{lem:khat}
Let $k$ be the function defined by \eqnref{functionk}. Then, $\widehat{k}(\xi)$ is even, decreasing in $\xi \ge 0$, continuously differentiable on $\Rbb$, $0 < \widehat{k}(\xi) \le \widehat{k}(0)=2$, and
\beq\label{khatdecay}
|\widehat{k}(\xi)| \lesssim \frac{1}{1+|\xi|^N}
\eeq
for any positive integer $N$.
\end{lemma}

\begin{proof}
Since
$$
\widehat{k}(\xi)= 2 \int_0^\infty \frac{\cos 2\pi \xi t} {(1+t^2)^{3/2}} dt,
$$
we see that $\widehat{k}$ is even, belongs to $C^1(\Rbb)$, and $\widehat{k}(0)=2$. If $|\xi| >1$, then integrations by parts show that $|\widehat{k}(\xi)| \lesssim |\xi|^{-N}$
for any positive integer $N$. Thus we have \eqnref{khatdecay}.

It now remains to prove that $\widehat{k}(\xi)$ is decreasing in $\xi \ge 0$. To prove it, we recall the relation
$$
\widehat{k}(\xi) = 2\pi \xi K_1(2\pi\xi),
$$
where $K_\nu$ denotes the modified Bessel function of the second kind (see \cite[10.32.11]{NIST}). Here $2\pi$ appears in the formula due to the definition \eqnref{Fourdef} of the Fourier transformation. Thanks to the recurrence relation $(\xi K_1(\xi))'=-\xi K_0(\xi)$ (\cite[10.29.2]{NIST}), we infer that $(\widehat{k})'(\xi)<0$ since $K_0(\xi) > 0$ if $\xi>0$. This completes the proof.
\end{proof}

\begin{lemma}\label{lem:Lhat}
$\widehat{L_0}$ is even, decreasing in $\xi \ge 0$, continuously differentiable on $\Rbb$, $0 < \widehat{L_0}(\xi) \le \widehat{L_0}(0)=1/2$, and for any $\Gd >0$
\beq\label{Lhatdecay}
|\widehat{L_0}(\xi)| \lesssim \frac{1}{1+|\xi|^{1-\Gd}}.
\eeq
\end{lemma}

\begin{proof}
It follows from Lemma \ref{lem:khat} that  $\widehat{L_0}$ is even, decreasing in $\xi \ge 0$, continuously differentiable on $\Rbb$, and $0 < \widehat{L_0}(\xi) \le \widehat{L_0}(0)=1/2$. To prove \eqnref{Lhatdecay},
suppose that $|\xi| >1$ and write
$$
\widehat{L_0}(\xi) = \frac{1}{4 \pi} \left( \int_0^{|\xi|^{-1+\Gd}} + \int_{|\xi|^{-1+\Gd}}^\pi \right) \widehat{k}(\sqrt{2 (1-\cos\Gt)} \xi) d\Gt =: I_1(\xi) + I_2(\xi).
$$
Since $0 < \widehat{k} \le 2$, we have $|I_1(\xi)| \lesssim |\xi|^{-1+\Gd}$.
Since $|\sqrt{2 (1-\cos\Gt)} \xi| \gtrsim |\xi|^{\Gd}$ if $|\xi|^{-1+\Gd} \le \Gt \le \pi$,
it follows from \eqnref{khatdecay} that $|I_2(\xi)| \lesssim |\xi|^{-N}$ for any positive integer $N$.
\end{proof}

\subsection{The NP operator on prolate spheroids and the limiting operator}

In this subsection, we prove that the limiting operator (as $R \to 0$) of the NP operator on prolate spheroids is the convolution by $L_0$ given in \eqnref{L0} on some test functions. We begin by  constructing test functions $g_\Gr$ for parameter $\Gr>0$. Eventually, we take  $\Gr=R^{1-\Gs}$ for some $\Gs \in (0,1)$. By Lemma \ref{lem:Lhat}, for $\Gl \in (0, 1/2]$  there is a unique point
$\xi_0 \in [0, \infty)$ such that
\beq
\label{GLdef-pro}
\Gl - \widehat{L_0}(\xi_0) =0 .
\eeq
Let $\Gz_1$ be a function on $\Rbb$ such that $\widehat{\Gz_1}$ is a non-negative compactly supported smooth function
with
$$
\int_{\Rbb} \widehat{\Gz_1}(\xi)d\xi=1.
$$
Then, $\Gr \widehat{\Gz_1}(\Gr (\xi-\xi_0))$ converges weakly to $\Gd_{\xi_0}(\xi)$, the one-dimensional Dirac delta function, as $\Gr \to \infty$. Let $\chi_1$ be a smooth cut-off function such that $\mbox{supp}(\chi_1) \subset B(0,1)$ and $\chi_1=1$
on $B(0,1/2)$. Define
\beq\label{gRx}
g_\Gr(x):= \Gr^{-1/2} e^{2\pi i \xi_0 x} (\chi_1\Gz_1)(\Gr^{-1}x).
\eeq

In this section $\| \ \|_2$ denotes the $L^2$-norm on $\Rbb^1$.

\begin{lemma}\label{lem:limit-pro}
Let $\Gl \in (0, 1/2]$ and $g_\Gr$ be defined by \eqnref{gRx} with $\xi_0$ satisfying \eqnref{GLdef-pro}.
Then the following hold:
\begin{itemize}
\item[{\rm (i)}] $\| g_\rho \|_{2} \sim 1$.
\item[{\rm (ii)}] $\| \Gl g_\rho - L_0 *g_\rho \|_{2} \to 0$ as $\rho \to \infty$.
\end{itemize}
\end{lemma}

\begin{proof} It is easy to see  $\| g_\rho \|_{2}=\|\chi_1\Gz_1\|_2 $, thus (i) follows. To show  (ii),
 we note that
$$
\mathcal F(\Gl g_\rho - L_0*g_\rho)=(\lambda- \widehat{L_0} (\xi)) \rho^{1/2} ( \widehat{\chi_1}\ast\widehat{\Gz_1}) (\Gr(\xi-\xi_0)).
$$
By Plancherel's theorem  and changing variables $\xi \to \rho^{-1}\xi+\xi_0$, we have
\[ \| \Gl g_\rho - L_0 *g_\rho \|_{2}^2= \int  |\lambda- \widehat{L_0} (\rho^{-1}\xi +\xi_0))|  ( \widehat{\chi_1}\ast\widehat{\Gz_1}) (\xi)|^2 d\xi. \]
Since $\lambda= \widehat{L_0} (\xi_0)$ and $\widehat{L_0}$ is continuous by Lemma \ref{lem:Lhat}, we obtain (ii) by the dominated convergence theorem.
\end{proof}

\begin{lemma}\label{lem:limit-pro2}
Let $\Hcal_{R}$ be the operator appearing in \eqnref{KcalHcal} and let $\Gr=R^{1-\Gs}$ for some $\Gs \in (0,1)$. Then,
\beq\label{2110}
 \|  \Hcal_R [g_\rho]- L_0\ast g_\rho \|_2\to 0,
\eeq
and hence
\beq
\label{L2}
\| \Gl g_\Gr  - \Hcal_{R}[g_\Gr] \|_{2} \to 0
\eeq
as $R \to \infty$.
\end{lemma}

\begin{proof}
By Lemma \ref{lem:limit-pro} (ii), \eqref{L2} is an immediate consequence of \eqnref{2110}.

To prove \eqnref{2110}, we break the integral kernel $H_R$ of $\Hcal_R$ as follows:
\beq
H_R(x_3,y_3)=H_R^u(x_3,y_3)+H_R^l(x_3,y_3),
\eeq
where
\begin{align*}
H_R^u(x_3,y_3) &=H_R(x_3,y_3)  \chi_{[R^{\Gs/2},\infty)}(|x_3-y_3| ), \\
H_R^l(x_3,y_3) &=H_R(x_3,y_3)  \chi_{[0, R^{\Gs/2})}(|x_3-y_3| ).
\end{align*}
Here and afterwards, $\chi_{[R^{\Gs/2},\infty)}$ and $\chi_{[0, R^{\Gs/2})}$ denote the characteristic functions of the intervals $[R^{\Gs/2},\infty)$ and $[0, R^{\Gs/2})$, respectively.

We also make a similar decomposition for $L_0$, namely,
\[ L_0:=L_0^u+ L_0^l,\]
where $L_0^l(t)=L_0^l(t)\chi_{[0, R^{\Gs/2})}(|t| ) $.  For notational convenience we denote by $\Hcal_R^l,$ $\Hcal_R^u$, $\Lcal^l_0,$ and  $\Lcal^u_0$  the operators defined by the integral kernels  $H_R^l,$ $H_R^u$, $L^l_0(x_3-y_3),$ and  $L^u_0(x_3-y_3)$, respectively. We then have
\beq
\Hcal_R [g_\rho]- L_0\ast g_\rho = (\Hcal_R^l- \Lcal^l_0) [g_\rho] + (\Hcal_R^u- \Lcal^u_0) [g_\rho].
\eeq

The term $(\Hcal_R^u- \Lcal^u_0) [g_\rho]$ is easy to handle. Indeed,   by \eqref{H-kernel} and \eqref{L0} it follows that
\[
0\le H_R(x_3,y_3), \ L_0(x_3-y_3)\lesssim |x_3-y_3|^{-3}.
\]
Hence, we have
\begin{align*}
&\int_{-\infty}^\infty |H_R^u(x_3,y_3)| +| L^u_0(x_3-y_3)| dy_3 \lesssim R^{-\Gs}.
\end{align*}
Young's inequality yields $\| \Hcal_R^u\|_{2\to 2}\lesssim R^{-\Gs}$ and $\| \Lcal_0^u\|_{2\to 2}\lesssim R^{-\Gs}.$
Here $\|\cdot\|_{2\to 2}$ denotes the operator norm from $L^2$ to $L^2$. Therefore,
$$
\| (\Hcal_R^u- \Lcal^u_0)[g_\rho]\|_2 \lesssim R^{-\Gs}\| g_\rho\|_2 \lesssim R^{-\Gs},
$$
where the last inequality holds thanks to Lemma \ref{lem:limit-pro} (i).

The matter is now reduced to showing
\beq\label{HL}
\| (\Hcal_R^l- \Lcal^l_0)[g_\rho]\|_2  \to 0
\eeq
as $R\to \infty$.  In order to prove this we further break the operator $\Hcal_R^l$ by decomposing its kernel. By $D(x_3,y_3,\Gt)$ we denote the denominator of the integrand of \eqref{H-kernel}, namely,
\[
D(x_3,y_3,\Gt)= 2\pi [\big( \Gn(x_3)^2+ \Gn(y_3)^2- 2\Gn(x_3)\Gn(y_3) \cos\Gt \big) + (x_3-y_3)^2]^{3/2},
\]
and break the numerator so that
\[ 1-R^{-2}x_3y_3- \Gn(x_3)\Gn(y_3) \cos\Gt= E_1(x_3,y_3,\Gt)+E_2(x_3,y_3,\Gt)+E_3(x_3,y_3,\Gt),
\]
where
\begin{align*}
  E_1(x_3,y_3)&:= 1-R^{-2}x_3y_3- \Gn(x_3)\Gn(y_3),
   \\
   E_2(x_3,y_3, \Gt)&:=(\Gn(x_3)\Gn(y_3)-1)(1- \cos\Gt),
   \\
  E_3(x_3,y_3,\Gt)&:=(1- \cos\Gt).
  \end{align*}
We then define
\[
H_R^{l,j} (x_3,y_3) = \chi_{[0, R^{\Gs/2})}(|x_3-y_3| ) \int_0^\pi  \frac{ E_j(x_3,y_3,\Gt)}{D(x_3,y_3,\Gt)}  d\Gt, \quad j=1,2,3,
\]
so that
\[
H_R^{l}=H_R^{l,1}+H_R^{l,2}+H_R^{l,3}.
\]
As before, we denote by $\Hcal_R^{l,j}$ the operator given  by the kernel $H_R^{l,j}$ for  $j=1, 2,3$.  For the proof of \eqref{HL},  we show the contributions from the operators $\Hcal_R^{l,2}$ and $\Hcal_R^{l,3}$ are negligible.
This reduces  \eqref{HL} to \eqref{Hl} below.

If $y_3$ lies in the support of $g_\Gr$, namely, $|y_3|\le R^{1-\Gs}$ and if $|x_3-y_3|\le R^{\Gs/2}$, then $|x_3| \lesssim R^{1-\Gs}$ if $0< \Gs \leq 1/2$. Thus, in order to show \eqref{HL}  we may assume
\beq\label{x3y3}
|y_3|\le R^{1-\Gs} \quad \mbox{and} \quad |x_3| \lesssim R^{1-\Gs}
\eeq
for the rest of this proof.
On the other hand, since
$$
\Gn(x_3)^2+ \Gn(y_3)^2- 2\Gn(x_3)\Gn(y_3) \cos\Gt =( \Gn(x_3)-\Gn(y_3))^2+  2\Gn(x_3)\Gn(y_3)(1- \cos\Gt),
$$
we have
\beq
D(x_3,y_3,\Gt)\ge C \big [(1-\cos\Gt)+ (x_3-y_3)^2\big]^{3/2}
\eeq
for some constant $C>0$.

Because of \eqnref{x3y3}, assuming $R$ is large enough,  one can easily   see that $
|E_1(x_3,y_3, \Gt)| \lesssim R^{-2}|x_3-y_3|^2.
$
Thus, we have
\[
|H_R^{l,1} (x_3,y_3)| \lesssim \frac{\chi_{[0, R^{\Gs/2})}(|x_3-y_3| )}{R^2}  \int_0^\pi  \frac{ |x_3-y_3|^2}{[ \Gt^2+ (x_3-y_3)^2]^{3/2}}  d\Gt
\lesssim \frac{\chi_{[0, R^{\Gs/2})}(|x_3-y_3| )}{R^2} ,
\]
which clearly yields
$$
\int_{-\infty}^\infty |H_R^{l,1} (x_3,y_3)|  dy_3 \lesssim R^{-2+\Gs/2}.
$$
Therefore, by symmetry and Young's inequality as before, we have
$$
\|\Hcal_R^{l,1} [g_\rho] \|_2 \lesssim  R^{-2+\Gs/2} .
$$

Thanks to \eqnref{x3y3}  it is easy to see $
|E_2(x_3,y_3, \Gt)| \lesssim R^{-2\Gs} (1- \cos\Gt)
$. So, we have
\begin{align*}
  |H_R^{l,2} (x_3,y_3)|
        &\lesssim R^{-2\Gs}\chi_{[0, R^{\Gs/2})}(|x_3-y_3| ) \int_0^\pi  \frac{ \Gt^2}{[\Gt^2+ (x_3-y_3)^2]^{3/2}}  d\Gt
        \\
        &\lesssim R^{-2\Gs}\chi_{[0, R^{\Gs/2})}(|x_3-y_3| ) (1+ |\log |x_3-y_3||).
\end{align*}
Thus, it follows that
$$
\int_{-\infty}^\infty |H_R^{l,2} (x_3,y_3)|  dy_3 \lesssim R^{-3\Gs/2}\log R.
$$
By symmetry and Young's inequality  this gives
$$
\|\Hcal_R^{l,2} [g_\rho] \|_2 \lesssim  R^{-3\Gs/2}\log R.
$$

The proof of \eqref{Hl} is now reduced to showing
\beq
\label{Hl}
\| (\Hcal_R^{l,3}- \Lcal^l_0)[g_\rho] \|_2  \to 0
\eeq
as $R\to \infty$. To prove this \eqref{Hl}, we first note
\beq\label{HLL}
|H_R^{l,3}(x_3,y_3)- L^l_0(x_3-y_3)| \lesssim  \chi_{[0, R^{\Gs/2})}(|x_3-y_3| ) \int_0^\pi (1-\cos\Gt) |K_\Gt(x_3,y_3)| d\Gt,
 \eeq
where
\begin{align*}
K_\Gt(x_3,y_3) &=\Big ( \Gn(x_3)^2+ \Gn(y_3)^2- 2\Gn(x_3)\Gn(y_3) \cos\Gt + (x_3-y_3)^2\Big) ^{-\frac 32} \\
& \qquad  - \Big ( 2-2\cos\Gt + (x_3-y_3)^2\Big) ^{-\frac 32}.
\end{align*}
Also, note that
\begin{align*}
&\Gn(x_3)^2+ \Gn(y_3)^2- 2\Gn(x_3)\Gn(y_3) \cos\Gt-  (2- 2\cos\Gt) \\
&=F( \frac{x_3}{R}, \frac{y_3}{R}) +  2(  \Gn(x_3)\Gn(y_3)-1) (1-\cos\Gt),
\end{align*}
where
\[
F(s,t):=2-s^2- t^2 -2\sqrt{(1- s^2)(1-t^2)}\,.
\]

If $R$ is sufficiently large, then
$
\Gn(x_3)^2+ \Gn(y_3)^2- 2\Gn(x_3)\Gn(y_3) \cos\Gt \ge 1-\cos\Gt .
$
Thus, by the mean value theorem, we have
\begin{align*}
|K_\Gt(x_3,y_3)|&\lesssim \frac{ | F( \frac{x_3}{R}, \frac{y_3}{R})|  +  |(  \Gn(x_3)\Gn(y_3)-1) (1-\cos\Gt)  | }{\big(1-\cos\Gt + (x_3-y_3)^2\big)^{5/2}}.
\end{align*}
Since $F$ is smooth on  $[-1/2, 1/2]\times [-1/2, 1/2]$, $\partial_s F(s,s)=0$ and $F(s,s)=0$, we have
$$
|F(s,t)|\le C|s-t|^2
$$
for  $s,t\in (-1/2, 1/2)$.  We then infer using  \eqnref{x3y3} that
\[
|K_\Gt(x_3,y_3)| \le C \frac{ R^{-2}| x_3-y_3 |^2 +  R^{-2\Gs} (1-\cos\Gt) }{\big(2-2\cos\Gt + (x_3-y_3)^2\big)^{5/2}}.
\]
Combining this with \eqref{HLL} yields
\begin{align*}
|H_R^{l,3}(x_3,y_3)- L^l_0(x_3-y_3)|
&\le C   \chi_{[0, R^{\Gs/2})}(|x_3-y_3| ) \int_0^\pi  \frac{ R^{-2}| x_3-y_3|^2\Gt^2 +  R^{-2\Gs} \Gt^4 }{\big(\Gt^2 + (x_3-y_3)^2\big)^{5/2}}d\Gt
\\
&\le C  \chi_{[0, R^{\Gs/2})}(|x_3-y_3| ) (R^{-2} +  R^{-2\Gs}(1+|\log |x_3-y_3||),
\end{align*}
from which it follows that
$$
\int |H_R^{l,3}(x_3,y_3)- L^l_0(x_3-y_3)| dy_3 \lesssim R^{-3\Gs/2}\log R.
$$
Because of symmetry, the integration with respect to $x_3$ satisfies the same inequality.  Therefore,  by Young's  inequality,
we infer
$$
\| (\Hcal_R^{l,3}- \Lcal^l_0)[g_\rho] \|_2  \lesssim R^{-3\Gs/2}\log R,
$$
which yields \eqref{Hl}.  This completes the proof.
\end{proof}

\subsection{Proof of Proposition \ref{prop:prolate}}\label{subsec:proofpro}

Let $\Gl \in (0, 1/2]$ and  $g_\Gr$ be defined  by \eqnref{gRx} where $\xi_0$ satisfies  \eqref{GLdef-pro}. Then, we  define the function $\psi_\Gr$ on $\p \GP_R$  by  \eqnref{extpro} with $g=g_\rho$.
Applying \eqnref{psig},
\eqnref{KcalHcal},  Lemma \ref{lem:limit-pro},  and Lemma \ref{lem:limit-pro2}, we see
\beq\label{limit-pro}
\lim_{R \to \infty} \frac{\| ( \Gl I - \Kcal_{\p\GP_R}) [\psi_\Gr] \|_{L^2(\p \GP_R)}}{\| \psi_\Gr \|_{L^2(\p \GP_R)}} =0.
\eeq

Let $R_j$ be a sequence such that $R_j \to \infty$ as $j \to \infty$. Suppose  $\Gl \notin \ol{\cup_{j=1}^\infty \Gs(\Kcal_{\p\GP_{R_j}})}$, then there is an $\Ge>0$ such that $[\Gl -\Ge, \Gl+\Ge] \cap \Gs(\Kcal_{\p\GP_{R_j}}) = \emptyset$ for all $j$. Therefore, there is a constant $C$ independent of $j$ such that
$$
C \| \Gvf \|_{L^2(\p\GP_{R_j})} \le \| ( \Gl I - \Kcal_{\p\GP_{R_j}}) [\Gvf] \|_{L^2(\p\GP_{R_j})}, \quad  \forall \Gvf \in L^2(\p\GP_{R_j})
$$
for all  $j$. This contradicts \eqnref{limit-pro}. Thus, we conclude  $\Gl \in \ol{\cup_{j=1}^\infty \Gs(\Kcal_{\p\GO_{R_j}})}$. This completes the proof.

\subsection{Proof of Theorem \ref{thm:prolate2}}

The (confocal) prolate spheroids can be canonically described in terms of the prolate spheroidal coordinates, which are given by
\begin{align*}
x_1 &=\frac{d}{2} [(\xi^2 -1)(1- \eta ^2)]^{1/2} \cos \phi,  \\
x_2 &=\frac{d}{2} [(\xi^2 -1)(1- \eta ^2)]^{1/2} \sin \phi,  \\
x_2 &= \frac{d}{2} \xi \eta,
\end{align*}
where $d$ is the distance between two foci, $(0,0,d/2)$ and $(0,0,-d/2)$, $1 \leq \xi <\infty$, $-1\leq \eta \leq 1$, and $\phi$ is the azimuthal angle lying in $[0, 2 \pi]$. The surfaces $\xi=L$ (constant) represent (confocal) prolate spheroids. The spheroid $\p\GP_R$ is $\xi=L$ with $R= L/\sqrt{L^2-1}$ after dilation by the factor of $(d/2)^2 (L^2-1)$. Let us  recall  that the NP spectrum is invariant under dilation. The limiting case $\xi=1$ is a degenerate case corresponding to the line segment between the foci, and $L=\infty$ is the sphere.

It was shown in \cite{AA} that NP eigenvalues on the surface $\xi =L>1$, or on $\p\GP_R$ with $R= L/\sqrt{L^2-1}$, are positive and given by  an explicit formula
\beq\label{ev:prolate}
\Gl_{m, n}(L) = \left( -\frac{1}{2} \right) (-1)^m \frac{(n-m)!}{(n+m)!} (L^2 -1)(P_n^m Q_n^m)' (L)
\eeq
for $n=1,2, \ldots$ and $m=-n, -n+1, \ldots, -1, 0, 1, \ldots, n$, where $P_n^m(L)$ and $Q_n^m(L)$ denote associated Legendre functions of the first kind and the second kind, respectively. On the surface $L=\infty$, namely,  the sphere,   we have
\beq\label{Glinfty}
\Gl_{m, n}(\infty)= \frac{1}{2(2n+1)}, \quad m=-n, -n+1, \ldots, -1, 0, 1, \ldots, n.
\eeq
Moreover, it is shown  in \cite{Mart} that $\Gl_{m, n}(L)$ enjoys the 1/2-property
\beq\label{a half property}
\sum_{m=-n}^n \Gl_{m, n}(L)=1/2 \quad\mbox{for all } n.
\eeq
It is worth mentioning that it  has not been known whether  NP eigenvalues on general surfaces do or do not satisfy the 1/2-property. We refer to \cite{AKMU} for a discussion on this.

We obtain the following proposition from (ii) of which Theorem \ref{thm:prolate2} immediately follows by taking $R_0= L_0/\sqrt{L_0^2-1}$. Note that (i) of the following proposition shows tunability of the eigenvalues by prolate spheroids, namely, for any $n$ and $\Gl \in (0, 1/2)$ there is a prolate spheroid characterized by $\xi=L$ such that $\Gl_{m, n}(L)=\Gl$ for some $m$. When $n=1$, it was proved in \cite{FK}.

\begin{prop}
\begin{itemize}
\item[{\rm (i)}] $\ds \cup_{L>1} \{\Gl_{m, n}(L): -n \le m \le n\} = (0, 1/2)$ for each $n=1,2, \ldots$.
\item[{\rm (ii)}] For any $L_0 >1$, $\ds \cup_{1<L\le L_0} \{\Gl_{0, n}(L): n=1,2,\ldots \} = (0, 1/2)$.
\end{itemize}
\end{prop}

\begin{proof}
From Rodrigues' formulas  which is also known as the Ivory--Jacobi formula (also see, e.g., \cite{AS})   we have the following:
\begin{align*}
P_n^0(z)&=:P_n(z) =\frac{1}{2^n n!} \frac{d^n (z^2-1)^n}{dz^n},
\\
Q_n^0(z)&=:Q_n(z) =\frac{1}{2} P_n(z) \log \frac{1+z}{1-z} - \sum_{m=1} ^n \frac{1}{m} P_{m-1}(z) P_{n-m} (z),
\\
P_n (z) Q_n (z) &=
\frac{1}{2} (P_n(z))^2  \log \frac{1+z}{1-z}
- P_n(z)\sum_{m=1} ^n \frac{1}{m} P_{m-1}(z) P_{n-m} (z) .
\end{align*}
Since $P_k$ is a polynomial (it is of degree $k$) for each $k$ and $P_n(1)=1$, we have
$$
(P_n Q_n)'(L)= \frac{1}{2(1-L)} + O(|\log(L-1)|)
$$
as $L \to 1+0$. We then infer from \eqref{ev:prolate} that $\Gl_{0, n}(L) \rightarrow \frac{1}{2}$ as $L \rightarrow 1+0$.

By \eqref{a half property}, we see that $\Gl_{m, n}(L) \to 0$ as $L \rightarrow 1$ if $m \neq 0$. Since each $\Gl_{m, n}(L)$ is continuous in $L$ on $(1, \infty)$,  by  \eqnref{Glinfty}  we have
$$
\bigcup_{L>1} \{\Gl_{m, n}(L): m \neq 0 \} \supseteq \left(0, \frac{1}{2(2n+1)} \right], \quad
\bigcup_{L>1} \{\Gl_{0, n}(L) \} \supseteq \left[\frac{1}{2(2n+1)},\frac{1}{2} \right)
$$
for each $n$.  Thus, (i) follows.

Since $\Gl_{m, n}(L) \rightarrow 0$ as $n \to \infty$ for each fixed $L$, we have $\Gl_{0, n}(L) \rightarrow 0$ as a subsequence, and hence (ii) follows.
\end{proof}

\section{Proof of Theorem \ref{thm:oblate}}\label{sec:3}

In this section we prove Theorem \ref{thm:oblate}. Again it is enough to consider the spectrum of $\Kcal_{\p \GO_{R}}$ on $L^2(\p\GO_R)$ since $\p\GO_R$ is smooth.

\subsection{Parametrization of the NP operator on oblate ellipsoids}

In this sections and those to follow, we use $X$ to represent points in $\Rbb^3$ saving $x$ for points in the plane, that is to say, $X=(x,x_3)=(x_1,x_2,x_3)$. Let $D_R$ be the projection of $\GO_R$ onto $x$-plane, namely,
$$
D_R:=\left \{ x=(x_1,x_2): \frac{x_1^2}{(a_1 R)^2} + \frac{x_2^2}{(a_2 R)^2} \le 1 \right \},
$$
and let
\beq
\Gg(x)=\Gg_R(x):=  \sqrt{1- \frac{x_1^2}{(a_1R)^2} - \frac{x_2^2}{(a_2R)^2}}, \quad x \in D_R.
\eeq
Then, $\p\GO_R$ consists of two pieces, namely, $\p\GO_R = \GG^+ \cup \GG^-$,
where
\beq\label{GOboundary}
\GG^{\pm} = \{ (x,\pm \Gg(x)): x \in D_R \},
\eeq
and the NP operator can be written as
$$
\Kcal_{\p\GO_{R}}[\Gvf](X)= \Big(\int_{\GG^+} + \int_{\GG^-}\Big) \frac{\la Y-X, \nu(Y) \ra}{4\pi |X-Y|^3} \Gvf(Y) dS(Y).
$$

Now, let us set
\begin{align}
\label{K1def}
K_R^1(x,y)= -\frac{1}{4\pi} \frac{\frac{1}{R^2\Gg(y)} \sum_{j=1}^2 (x_j-y_j) \frac{y_j}{a_j^2} + (\Gg(x) - \Gg(y)) } { \big[ |x-y|^2 + (\Gg(x) - \Gg(y))^2 \big]^{3/2} },
\\[4pt]
\label{K2def}
K_R^2(x,y)= -\frac{1}{4\pi} \frac{\frac{1}{R^2\Gg(y)} \sum_{j=1}^2 (x_j-y_j) \frac{y_j}{a_j^2} - (\Gg(x) + \Gg(y)) } { \big[ |x-y|^2 + (\Gg(x) +\Gg(y))^2 \big]^{3/2} }.
\end{align}
For $Y=(y,y_3) \in \GG^\pm$,   we have
$$
\nu(Y) dS(Y)= \Big( \frac{y_1}{(a_1R)^2 \Gg(y)}, \frac{y_2}{(a_2R)^2 \Gg(y)}, \pm 1 \Big) dy.
$$
Therefore, for $(x, \Gg(x)) \in \GG^+$ we  obtain
\begin{align}
\Kcal_{\p\GO_{R}}[\Gvf](x, \Gg(x)) &= \int_{D_R} K_R^{1}(x,y) \Gvf^+(y) dy + \int_{D_R} K_R^{2}(x,y) \Gvf^-(y) dy \nonumber \\
&=: \Kcal_{R}^{1}[\Gvf^+](x) + \Kcal_{R}^{2}[\Gvf^-](x), \label{upperform}
\end{align}
where
$$
\Gvf^+(y):= \Gvf(y, \Gg(y)), \quad \Gvf^-(y):= \Gvf(y, -\Gg(y)).
$$
Similarly, one can easily  see
\beq\label{lowerform}
\Kcal_{\p\GO_{R}}[\Gvf](x, -\Gg(x)) = \Kcal_{R}^{2}[\Gvf^+](x) + \Kcal_{R}^{1}[\Gvf^-](x),  \quad (x, -\Gg(x)) \in \GG^-.
\eeq

The surface measure on $\GG^\pm$ is given by
\beq
dS(x)= \Go_R(x) dx,
\eeq
where
\beq
\Go_R(x)= \frac{\chi_{D_R}(x)}{\Gg(x)} \sqrt{\frac{x_1^2}{(a_1R)^4} + \frac{x_2^2}{(a_2R)^4} + \Gg(x)^2} \, .
\eeq
For a measurable subset $U$ of $D_R$, we set
\beq
\Go_R(U) := \int_U \Go_R(x) dx.
\eeq
The following elementary lemma will be used later.

\begin{lemma}\label{lem:ARBR}
Let
\beq
A_R:= \Big\{ x \in D_R: \frac{x_1^2}{(a_1R)^4} + \frac{x_2^2}{(a_2R)^4} \le \Gg(x)^2 \Big\} , \quad B_R:= D_R \setminus A_R.
\eeq
The following hold:
\begin{itemize}
\item[{\rm (i)}] $\Go_R \ge 1$ on $D_R$ and $\Go_R \sim 1$ on $A_R$.
\item[{\rm (ii)}] $\Go_R(A_R) \sim R^2$.
\item[{\rm (iii)}] $\Go_R(B_R) \sim 1$.
\end{itemize}
\end{lemma}

\begin{proof}
 (i) is clear since  $\Go_R \ge 1$ on $D_R$ and $1 \le \Go_R \le \sqrt{2}$ on $A_R$.

Let $m=\min\{a_1, a_2\}$ and $M=\max\{a_1, a_2\}$. If $x \in A_R$, then
\beq\label{BRset}
\frac{x_1^2}{(a_1R)^2} + \frac{x_2^2}{(a_2R)^2} < C_1,
\eeq
where $C_1:= (MR)^2/((MR)^2+1)$, and hence
$
\Go_R(A_R) \lesssim R^2.
$
Likewise, we have
\beq\label{BRset2}
C_2 \le \frac{x_1^2}{(a_1R)^2} + \frac{x_2^2}{(a_2R)^2} <1,  \quad x \in B_R,
\eeq
where $C_2:= (mR)^2/((mR)^2+1)$. After the changes of variables $x_j=a_j R y_j$ ($j=1,2$), we have
$$
\Go_R(B_R) \lesssim R^2 \int_{V} \frac{1}{\sqrt{1-|y|^2}} \sqrt{\frac{y_1^2}{(a_1R)^2} + \frac{y_2^2}{(a_2R)^2}} \, dy \lesssim
R \int_{V} \frac{|y|}{\sqrt{1-|y|^2}} dy,
$$
where $V=\{ \sqrt{C_2} \le |y| <1 \}$. It thus follows that
$$
\Go_R(B_R) \lesssim
R \int_{\sqrt{C_2}}^1 \frac{r}{\sqrt{1-r^2}} dr \sim 1.
$$
On the other hand, it is clear that $\Go_R(B_R) \gtrsim 1$ if $R$ is large enough.  This shows ${\rm (iii)}$.
Since
$$
\Go_R(A_R) + \Go_R(B_R) = \Go_R(D_R) \gtrsim R^2,
$$
we have (ii). This completes the proof.
\end{proof}

For a function $f$ defined on $D_R$, we denote
\beq
\|f\|_{\Go_R}= \| f\|_{L^2(D_R, \Go_R)}.
\eeq
Let $f^\pm$ be functions on $\GG^\pm$ defined by $f^\pm(x, \pm \Gg(x))=f(x)$. Then, we have
\beq
\| f^+ \|_{L^2(\GG^+)} = \| f^- \|_{L^2(\GG^-)} = \|f\|_{\Go_R}.
\eeq

\subsection{The NP operator on oblate ellipsoids and the Poisson integral}

Since $\Gg(y)$ tends to $1$ pointwise as $R \to \infty$, one can expect from \eqnref{K1def} and \eqnref{K2def} that $K_R^1(x,y)$ and $K_R^2(x,y)$ respectively tend to $0$ and $\frac{1}{2}P_2(x-y)$ (if $x \neq y$) as $R \to \infty$, where $P_t(x)$ ($x \in \Rbb^2$) is the Poisson kernel
\beq\label{Pkernel}
P_t(x) = \frac{1}{2\pi} \frac{t}{(|x|^2+t^2)^{\frac{3}{2}}},   \quad (x,t)\in \mathbb R^2\times \mathbb R_+.
\eeq

We now construct test functions $f_\Gr$ in a similar manner as  $g_\Gr$ in the previous section. Recall that $\widehat{P_t}(\xi) = \exp(-2\pi t |\xi|)$.
For $\Gl \in (0, 1/2]$, choose $\xi_0 \in \Rbb^2$ such that
\beq\label{GLdef}
\Gl - \frac{1}{2} e^{-4\pi |\xi_0|} =0.
\eeq
Let $\Gz$ be a function on $\Rbb^2$ such that $\widehat{\Gz}$ is a non-negative compactly supported smooth function satisfying
$$
\int_{\Rbb^2} \widehat{\Gz}(\xi)d\xi=1.
$$
Then, $\Gr^2 \widehat{\Gz}(\Gr (\xi-\xi_0))$ converges weakly to $\Gd_{\xi_0}(\xi)$ as $\Gr \to \infty$.

Let $m=\min\{a_1, a_2\}$ as before and let $\chi$ be a smooth cut-off function such that $\mbox{supp}(\chi) \subset B(0,m)$ and $\chi=1$ on $B(0,m/2)$, where $B(0,r)$ denotes the disk centered at the origin of radius $r$. Define
\beq\label{fRx}
f_\Gr(x):= \Gr^{-1} e^{2\pi i \xi_0 x} (\chi\Gz)(\Gr^{-1}x).
\eeq
Note that $f_\Gr$ is supported in $D_\rho:=B(0,\rho m)$ and
\beq\label{fhatrho}
\widehat{f_\Gr}(\xi)= \Gr \widehat{(\chi\Gz)}(\Gr(\xi-\xi_0)).
\eeq

\begin{lemma}\label{lem:limit}
Let $\Gl \in (0, 1/2]$ and $f_\Gr$ be defined by \eqnref{fRx} with $\xi_0$ satisfying \eqnref{GLdef}. If $\rho=R^{1-\Gs}$ for some $\Gs >0$, then the following hold:
\begin{itemize}
\item[{\rm (i)}] $\| f_\rho \|_{\Go_R} \sim 1$.
\item[{\rm (ii)}] $\| \Gl f_\rho -\frac{1}{2}P_2 *f_\rho \|_{\Go_R} \to 0$ as $R \to 0$.
\end{itemize}
\end{lemma}

\begin{proof}
Since $f_\Gr$ is supported in $D_\Gr$ and $D_\Gr \subset A_R$ with $\rho=R^{1-\Gs}$, it follows from (i) of Lemma \ref{lem:ARBR}  that
$$
\| f_\rho \|_{\Go_R} \sim \| f_\rho \|_{2},
$$
where $\| \ \|_2$ denotes the $L^2$-norm on $\Rbb^2$ with respect to the Lebesgue measure.  Using Plancherel's theorem  and  \eqnref{fhatrho}, we have
$
\| f_\Gr \|_{2}^2 = \int_{\Rbb^2} |\widehat{f_\Gr}(\xi)|^2 d\xi = \int_{\Rbb^2} |\widehat{(\chi\Gz)}(\xi)|^2 d\xi.
$
Thus, we get  (i)  combining this with the above.

By \eqnref{BRset2}, we have $|x| \sim R$ for $x \in B_R$. Thus, if $x \in B_R$ and $y \in D_\Gr$ with $\Gr=R^{1-\Gs}$, then
$
|x-y| \gtrsim R-\Gr \gtrsim R.
$
So, we have $
P_2(x-y) \lesssim R^{-3}
$
for $x \in B_R$ and $y \in D_\Gr$.
Since  the support of  $f_\Gr$ is contained  in  $ A_R$,  by   (ii) in Lemma \ref{lem:ARBR}  and  H\"older inequality we have
$$
|(P_2 * f_\Gr)(x)| \lesssim R^{-2} \| f_\Gr \|_2 \lesssim R^{-2}.
$$
Combining this and  (iii) in Lemma \ref{lem:ARBR}  we obtain
\beq\label{BRint1}
\int_{B_R} |(P_2 * f_\Gr)(x)|^2 \Go_R(x) dx \lesssim R^{-4}.
\eeq

We now  proceed to  prove (ii). Note that
$$
\big\| \Gl f_\Gr -\frac{1}{2}P_2 *f_\Gr \big\|_{\Go_R}^2 = \big\| \Gl f_\Gr -\frac{1}{2}P_2 *f_\Gr \big\|_{L^2(A_R, \Go_R)}^2 + \big\| \frac{1}{2}P_2 *f_\Gr \big\|_{{L^2(B_R, \Go_R)}}^2.
$$
Thanks to   (i)  of Lemma \ref{lem:ARBR}  and \eqnref{BRint1}, it suffices to show that
\beq\label{forii}
\big\| \Gl f_\Gr -\frac{1}{2}P_2 *f_\Gr \big\|_{2}^2 \to 0 \quad\mbox{as } R \to \infty.
\eeq
By \eqnref{fhatrho}, we have
$$
\Fcal \Big( \Gl f_\Gr -\frac{1}{2}P_2 *f_\Gr \Big)(\xi)= \Big( \Gl - \frac{1}{2} e^{-4\pi |\xi|} \Big) \Gr \widehat{(\chi\Gz)}(\Gr(\xi-\xi_0)).
$$
Hence, changing variables $\xi\to \xi/\rho$,  we see
\[\big\| \Gl f_\Gr -\frac{1}{2}P_2 *f_\Gr \big\|_{2}^2=\int_{\Rbb^2} \big| \Gl - \frac{1}{2} e^{-4\pi | \frac{\xi}{\Gr} + \xi_0 |} \big|^2 \big| \widehat{(\chi\Gz)}(\xi) \big|^2 d\xi.\]
We  then break the right hand side  as follows:
\begin{align*}
{\rm I + I\!I}:=\Big(\int_{|\xi| \le \sqrt{\Gr}} + \int_{|\xi| > \sqrt{\Gr}}  \Big)  \big| \Gl - \frac{1}{2} e^{-4\pi | \frac{\xi}{\Gr} + \xi_0 |} \big|^2 \big| \widehat{(\chi\Gz)}(\xi) \big|^2 d\xi.
\end{align*}
If $|\xi| \le \sqrt{\Gr}$, we have from \eqnref{GLdef}
$
\big| \Gl - \frac{1}{2} e^{-4\pi | \frac{\xi}{\Gr} + \xi_0 |} \big| \lesssim \frac{|\xi|}{\Gr}.
$
Thus, $
\mathrm I \lesssim \Gr^{-1} $, so $\mathrm I\to 0$ as $R\to \infty$.
Since $\chi\Gz$ is compactly supported and smooth, we have
$
\big| \widehat{(\chi\Gz)}(\xi) \big| \lesssim (1+|\xi|)^{-N}
$
for any $N$. So, we have
$$
{\rm I\!I} \lesssim \int_{|\xi| > \sqrt{\Gr}} ( 1+ |\xi| )^{1-2N} d\xi \lesssim \Gr^{2-N}.
$$
As a results, ${\rm I\!I} \to 0$ as $R\to \infty$. Therefore, we conclude \eqnref{forii}.
\end{proof}

\begin{lemma}\label{lem:limit2}
Let $\Gl \in (0, 1/2]$ and $f_\Gr$ be defined by \eqnref{fRx} with $\xi_0$ satisfying \eqnref{GLdef}. Suppose  $\Kcal_{R}^{1}$ and  $\Kcal_{R}^{2}$  are given by \eqref{upperform} and  $\rho=R^{1-\Gs}$ for some $\Gs >0$, then the following hold:
\begin{itemize}
\item[{\rm (i)}] $\| \Kcal_{R}^{1}[f_\rho] \|_{\Go_R} \to 0$ as $R \to \infty$.
\item[{\rm (ii)}] $\| \Gl f_\rho - \Kcal_{R}^{2}[f_\rho] \|_{\Go_R} \to 0$ as $R \to \infty$.
\end{itemize}

\end{lemma}

\begin{proof}
As in the proof of Lemma \ref{lem:limit}, we note that $|x-y| \gtrsim R$ if $x \in B_R$ and $y \in D_\Gr$ with $\Gr=R^{1-\Gs}$. Since $\Gg(y) \gtrsim 1$, using  \eqref{K1def}  and \eqref{K2def} we have
$$
|K_R^j (x,y)| \lesssim |x-y|^{-2} \frac{\Gr}{R} \lesssim R^{-2-\Gs}.
$$
Since $\|f_\rho\|_1\lesssim \rho$,  $|\Kcal_R^j [f_\Gr]|\lesssim   R^{-1}$.
By  (iii)  of Lemma \ref{lem:ARBR}  we obtain
\beq\label{BRint2}
\int_{B_R} |\Kcal_R^j [f_\Gr](x)|^2 \Go_R(x) dx \lesssim R^{-2}, \quad j=1,2.
\eeq
Thus, in order to prove (i), it suffices to prove
\beq\label{foriii}
\|\Kcal_{R}^{1}[f_\rho]\|_{L^2(A_R)}^2 \to 0 \quad\mbox{as } R \to \infty.
\eeq

Since $\Gr=R^{1-\Gs}$, changing variables $x\to Rx$ and $y\to Ry$ yields
\begin{align}
\label{KR1}
\|\Kcal_{R}^{1}[f_\rho]\|_{L^2(A_R)}^2
&= R^{-2+2\Gs} \int_{A_1}\Big| \int G^1_R(x,y)  e^{2\pi i R\xi_0\cdot y} (\chi\Gz)( R^\Gs y)  dy\Big|^2 dx,
 \end{align}
where
\beq
\label{G1R}
G_R^1(x,y):= R^3 K_R^1(Rx,Ry)=
-\frac{1}{4\pi} \frac{\Gg_1(x) - \Gg_1(y)+ \frac{1}{\Gg_1(y)} \sum_{j=1}^2 (x_j-y_j) \frac{y_j}{a_j^2} } { \big[ |x-y|^2 + R^{-2}(\Gg_1(x) - \Gg_1(y))^2 \big]^{3/2} }.
\eeq

Note that $(\chi\Gz)( R^\Gs y)$ is supported in $B(0, m R^{-\Gs})$ and  $\p_j\Gg_1(y)=-\frac{1}{\Gg_1(y)} \frac{y_j}{a_j^2} $.  Thus,
if $y \in B(0, m R^{-\Gs})$ and $x \in B(0, m/2)$,  by Taylor's theorem we have
$$
\Gg_1(x) - \Gg_1(y)+ \frac{1}{\Gg_1(y)} \sum_{j=1}^2 (x_j-y_j) \frac{y_j}{a_j^2} = \frac{1}{2} \sum_{i,j=1}^2 (\p_i\p_j \Gg_1)(x_*) (x_i-y_i)(x_j-y_j)
$$
for some $x_* \in B(0, m/2)$.

Since $|(\p_i\p_j \Gg_1)(x)| \lesssim 1$ for $x \in B(0, m/2)$, we have
$$
\Big|\Gg_1(x) - \Gg_1(y)+ \frac{1}{\Gg_1(y)} \sum_{j=1}^2 (x_j-y_j) \frac{y_j}{a_j^2}  \Big| \lesssim |x-y|^2,
$$
and hence
\beq\label{3000}
|G^1_R(x,y)|\lesssim |x-y|^{-1}.
\eeq
If $y \in B(0, m R^{-\Gs})$ and $x \in A_1 \setminus B(0, m/2)$, then $|x-y| \gtrsim 1$. Thus, we
see from \eqref{G1R} that  \eqnref{3000} remains to be valid for  $y \in B(0, m R^{-\Gs})$ and $x \in A_1$.
Therefore,  using \eqref{3000} we have
$$
\sup_{x \in A_1} \int_{B(0, m R^{-\Gs})} |G^1_R(x,y)|^p dy \lesssim 1, \quad 1\le p<2.
$$
Taking $p=3/2$, by H\"older's inequality  we have
$$
\sup_{x \in A_1} \Big| \int G^1_R(x,y)  e^{2\pi i R\xi_0 x} (\chi\Gz)( R^\Gs y)  dy\Big| \lesssim R^{-2\Gs/3}.
$$
Combining this and \eqref{KR1} we thus obtain
\beq\label{2000}
\|\Kcal_{R}^{1}[f_\rho]\|_{L^2(A_R)} \lesssim R^{-1+2\Gs/3},
\eeq
 which yields \eqnref{foriii} if we take a  $\sigma>0$ small enough.

We now show (ii). We prove
$$
\| \Kcal_{R}^{2}[f_\rho] - \frac{1}{2} P_2 * f_\rho\|_{\Go_R} \to 0 \quad\mbox{as } R \to \infty.
$$
Then (ii) follows by Lemma \ref{lem:limit} (ii). As before,  thanks to \eqnref{BRint1} and \eqnref{BRint2}  it suffices  to show
\beq\label{foriv}
\Big\|\Kcal_{R}^{2}[f_\rho] - \frac{1}{2} P_2 * [f_\rho]\Big\|_{L^2(A_R)} \to 0 \quad\mbox{as } R \to \infty.
\eeq

Let us set
\[ K_R^{2,e}(x,y)= \frac{1}{2\pi} \frac{\Gg(x)}{\big[ |x-y|^2 + (\Gg(x) +\Gg(y))^2 \big]^{3/2} }.\]
We note from \eqref{K1def} and \eqref{K2def} that
$K_{R}^{2}= K_{R}^{1}+   K_R^{2,e}$.  Thus, we may decompose
$$
\Kcal_{R}^{2}[f_\rho]= \Kcal_{R}^{1}[f_\rho]+ \Kcal_{R}^{2,e}[f_\rho],
$$
where $\Kcal_{R}^{2,e}$ denotes  the operator defined by the integral kernel  $K_R^{2,e}$.  Thus,
by \eqref{2000} proving \eqnref{foriv} is reduced to proving
\beq\label{foriv2}
\Big\|\Kcal_{R}^{2,e}[f_\rho] - \frac{1}{2} P_2 * [f_\rho]\Big\|_{L^2(A_R)} \to 0 \quad\mbox{as } R \to \infty.
\eeq

For simplicity we set
\begin{align*}
J(x)&= \Kcal_{R}^{2,e}[f_\rho](x) - \frac{1}{2} P_2 * [f_\rho](x),
\\
K(x,y)&= K_R^{2,e}(x,y) - \frac{1}{2} P_2(x-y).
\end{align*}
So, we have
$$
J(x) = \int K(x,y) f_\Gr(y) dy.
$$

 Let $\Gs'$ be a number to be determined later, but satisfying $0< \Gs' <\Gs$. If $x \in A_R \setminus B(0, R^{1-\Gs'})$ and $y$ is in the support of $f_\Gr$, namely, $y \in B(0, m R^{1-\Gs})$, then
$
|x-y| \gtrsim R^{1-\Gs'}.
$
Hence, we have
$$
|K(x,y)| \lesssim |K_R^{2,e}(x,y)| + |P_2(x-y)| \lesssim R^{-3+3\Gs'}.
$$
By  (i) of  Lemma \ref{lem:limit}   we see that
$$
|J(x)| \lesssim R^{-3+3\Gs'} |B(0, m R^{1-\Gs})|^{1/2} \| f_\Gr \|_2 \lesssim R^{-2+3\Gs'-\Gs}.
$$
Therefore, we have
$$
\int_{A_R \setminus B(0, R^{1-\Gs'})} |J(x)|^2 dx \le R^{6\Gs'-2\Gs-2}.
$$
We choose $\Gs'$ so that $3\Gs' <\Gs$. Then we see
\beq\label{2300}
\int_{A_R \setminus B(0, R^{1-\Gs'})} |J(x)|^2 dx \to 0 \quad\mbox{as } R \to \infty.
\eeq

To handle the remaining part we only need to consider  $y \in B(0, m R^{1-\Gs})$ and $x \in B(0, R^{1-\Gs'})$. We write
\begin{align*}
2\pi K(x,y) &= \frac{\Gg(x)-1}{\big[ |x-y|^2 + (\Gg(x) +\Gg(y))^2 \big]^{3/2} } \\
& \qquad + \Big( \frac{1}{\big[ |x-y|^2 + (\Gg(x) +\Gg(y))^2 \big]^{3/2}} - \frac{1}{\big[ |x-y|^2 + 2^2 \big]^{3/2}} \Big).
\end{align*}
Since $|\Gg(x)-1| \lesssim R^{-2}|x|^2 \lesssim R^{-2\Gs'}$ for $x \in B(0, R^{1-\Gs'})$, one can easily see that  the absolute values of  the first and the second terms in the right hand side  are respectively bounded by $R^{-2\Gs'} k_3(x-y)$ and $R^{-2\Gs'} k_5(x-y)$ for $y \in B(0, m R^{1-\Gs})$ and $x \in B(0, R^{1-\Gs'})$, where we denote
$$
k_n(x)= \frac{1}{(|x|^2+1)^{n/2}}.
$$
Thus, we have  $
|J(x)| \lesssim R^{-2\Gs'} (k_3 * |f_\Gr| + k_5 * |f_\Gr|)
$
for $x \in B(0, R^{1-\Gs'})$. Therefore,
$$
\int_{B(0, R^{1-\Gs'})} |J(x)|^2 dx \lesssim R^{-4\Gs'} (\| k_3 * |f_\Gr| \|_2^2 + \| k_5 * |f_\Gr| \|_2^2).
$$
Since $k_n$ ($n>2$) is integrable,  applying Young's convolution inequality  we obtain
$$
\int_{B(0, R^{1-\Gs'})} |J(x)|^2 dx \lesssim R^{-4\Gs'} \| f_\Gr \|_2^2 \lesssim R^{-4\Gs'}.
$$
This together with \eqnref{2300} yields \eqnref{foriv}.  So, the proof is completed.
\end{proof}

\subsection{Proof of Theorem \ref{thm:oblate}}

Let $\Gl \in [-1/2,1/2]\setminus \{0\}$. Let $f_\Gr$ denote the function given  by  \eqnref{fRx}  with $\Gr=R^{1-\Gs}$ for some $\Gs \in (0,1)$
where $\xi_0$ is given by \eqref{GLdef} with  $\lambda$ replaced by  $|\Gl|$.  Though $f_\rho$ is sightly different from the previous one, we keep using the same notation.

We now define $\Gvf_\Gr$ on $\p\GO_R$.
If  $\Gl \in (0, 1/2]$,
\beq\label{Gvfdef3D}
\Gvf_\Gr(X)= \Gvf_\Gr(x, x_3) :=
f_\Gr(x), \quad X \in \GG^+ \cup \GG^-.
\eeq
If $\Gl \in [-1/2,0)$, we define
\beq\label{fodef}
\Gvf_\Gr(X) :=
\begin{cases}
f_\Gr(x) \quad &\mbox{if } X \in \GG^+ , \\
-f_\Gr(x) \quad &\mbox{if } X \in \GG^-.
\end{cases}
\eeq

For $\Gl \in (0, 1/2]$,  it follows from \eqnref{upperform} and \eqnref{lowerform} that
$$
\Kcal_{\p\GO_R}[\Gvf_\Gr](X)= \Kcal_{R}^{1}[f_\Gr](x) + \Kcal_{R}^{2}[f_\Gr](x)
, \quad X \in \GG^+ \cup \GG^-.
$$
As a result, we have
$$
\Gl \Gvf_\Gr(X) - \Kcal_{\p\GO_R}[\Gvf_\Gr](X)= -\Kcal_{R}^{1}[f_\Gr](x) + (\Gl f_\Gr - \Kcal_{R}^{2}[f_\Gr])(x), \quad X \in \GG^+ \cup \GG^-.
$$
When  $\Gl \in [-1/2,0)$, we similarly  have
$$
\Kcal_{\p\GO_R}[\Gvf_\Gr](X)=
\begin{cases}
\Kcal_{R}^{1}[f_\Gr](x) - \Kcal_{R}^{2}[f_\Gr](x) \quad &\mbox{if } X \in \GG^+ , \\
-\Kcal_{R}^{1}[f_\Gr](x) + \Kcal_{R}^{2}[f_\Gr](x) \quad &\mbox{if } X \in \GG^-,
\end{cases}
$$
and, consequently,
$$
\Gl \Gvf_\Gr(X) - \Kcal_{\p\GO_R}[\Gvf_\Gr](X)=
\begin{cases}
-\Kcal_{R}^{1}[f_\Gr](x) + (\Gl f_\Gr(x) + \Kcal_{R}^{2}[f_\Gr](x)) \quad &\mbox{if } X \in \GG^+ , \\
\Kcal_{R}^{1}[f_\Gr](x) - (\Gl f_\Gr(x) + \Kcal_{R}^{2}[f_\Gr](x)) \quad &\mbox{if } X \in \GG^-.
\end{cases}
$$
In  either case, we thererfore  have
$$
\big \| \Gl \Gvf_\Gr - \Kcal_{\p\GO_R}[\Gvf_\Gr] \big \|_{L^2(\p\GO_R)} \le \big \| \Kcal_{R}^{1}[f_\Gr] \big \|_{\Go_R} + \big \| |\Gl| f_\Gr - \Kcal_{R}^{2}[f_\Gr] \big \|_{\Go_R}.
$$
Since $\| \Gvf_\Gr \|_{L^2(\p\GO_{R})} \sim \| f_\Gr \|_{\Go_R}$, we obtain the next proposition as an immediate consequence of Lemma \ref{lem:limit2}.

\begin{prop}\label{prop:oblate}
If $\Gl \in [-1/2,0) \cup (0, 1/2]$, then
\beq\label{limit}
\lim_{R \to \infty} \frac{\| ( \Gl I - \Kcal_{\p\GO_R}) [\Gvf_\Gr] \|_{L^2(\p\GO_{R})}}{\| \Gvf_\Gr \|_{L^2(\p\GO_{R})}} =0.
\eeq
\end{prop}

Once we have this proposition,
the rest of the proof  of Theorem \ref{thm:oblate} is the same as that of Proposition \ref{prop:prolate} in subsection \ref{subsec:proofpro}. So, we omit the detail.

\section{Proof of Theorem \ref{thm:thinflat}}\label{sec:4}

For $\GS^\pm$ and $\GS^s$ given in \eqnref{Fboundary}, we define
$$
\GS_R^\pm= \{(Rx, x_3): X=(x, x_3) \in \GS^\pm \}
$$
and $\GS_R^s$ likewise. Then, we have
$$
\p \GF_R = \GS_R^+ \cup \GS_R^- \cup \GS_R^s.
$$
Let $U_R= \{Rx: x\in U \}$, which is the projection of $\GS_R^\pm$ onto the $x$-plane.

If a function $\Gvf$ defined on $\p \GF_R$ is supported in $\GS_R^+ \cup \GS_R^-$, then we write
\begin{align*}
\Kcal_{\p\GF_R} [\Gvf](x, x_3) & = -\frac{1}{4\pi} \int_{\Rbb^2} \frac{x_3-1}{[|x-y|^2 + (x_3-1)^2]^{3/2}} \Gvf^+(y) dy \\
&\quad + \frac{1}{4\pi} \int_{\Rbb^2} \frac{x_3+1}{[|x-y|^2 + (x_3+1)^2]^{3/2}} \Gvf^-(y) dy,
\end{align*}
where $\Gvf^\pm(y)=\Gvf(y, \pm 1)$. Thus, if $x \in U_R$, then
\beq\label{4100}
\Kcal_{\p\GF_R} [\Gvf](x, \pm 1) = \frac{1}{2} (P_2 * \Gvf^\pm)(x).
\eeq

Let $f_\Gr$ be the function defined by \eqnref{fRx} (the number $m$ used to define $f_\Gr$ is chosen so that $B(0,m) \subset U$ in this case). By slightly modifying the proof of Lemma \ref{lem:limit}, one can prove the following lemma. Note that  we here use $H^{1/2}$ norm since $\p\GF_R$ is allowed to be Lipschitz continuous.

\begin{lemma}\label{lem:limitGF}
Let $\Gl \in (0, 1/2]$ and $f_\Gr$ be defined by \eqnref{fRx} with $\xi_0$ satisfying \eqnref{GLdef}. The following hold:
\begin{itemize}
\item[{\rm (i)}] $\| f_\rho \|_{H^{1/2}(\Rbb^2)} \sim 1$.
\item[{\rm (ii)}] $\| \Gl f_\rho -\frac{1}{2}P_2 *f_\rho \|_{H^{1/2}(\Rbb^2)} \to 0$ as $\Gr \to 0$.
\end{itemize}
\end{lemma}

Let $\Gl \in [-1/2,1/2]$ ($\Gl \neq 0$) and let $f_\Gr$ be the function defined in \eqnref{fRx} corresponding to $|\Gl|$. Let $\Gr=R^{1-\Gs}$ for some $\Gs \in (0,1)$.  In the same manner as  \eqnref{Gvfdef3D} and \eqnref{fodef}, we define $\Gvf_\Gr$ on $\p\GF_R$:
\begin{align}
\label{GvfGF}
\Gvf_\Gr(X)&= \Gvf_\Gr(x, x_3) :=
\begin{cases}
f_\Gr(x) \qquad &\mbox{if } X \in \GS_R^+ \cup \GS_R^-, \\
0 \qquad &\mbox{if } X \in \GS_R^s,
\end{cases}
\qquad \Gl \in (0, 1/2],
\\[4pt]
\label{foGF}
\Gvf_\Gr(X)&= \Gvf_\Gr(x, x_3) :=
\begin{cases}
f_\Gr(x) \quad &\mbox{if } X \in \GS_R^+ , \\
-f_\Gr(x) \quad &\mbox{if } X \in \GS_R^-, \\
0 \quad &\mbox{if } X \in \GS_R^s,
\end{cases}
\qquad\qquad \ \Gl \in [-1/2,0).
\end{align}

The following proposition which is analogous to Proposition \ref{prop:oblate} yields Theorem \ref{thm:thinflat} in the same way as Proposition \ref{prop:oblate} yields Theorem \ref{thm:oblate}.

\begin{prop}\label{prop:thinflat}
Let $\Gr=R^{1-\Gs}$ for some $\Gs \in (0,1)$. If $\Gl \in [-1/2,0) \cup (0, 1/2]$, then
\beq\label{limitGF}
\lim_{R \to \infty} \frac{\| ( \Gl I - \Kcal_{\p\GF_R}) [\Gvf_\Gr] \|_{H^{1/2}(\p\GF_{R})}}{\| \Gvf_\Gr \|_{H^{1/2}(\p\GF_{R})}} =0.
\eeq
\end{prop}

\begin{proof}
By (i) in  Lemma \ref{lem:limitGF},
$
\| \Gvf_\Gr \|_{H^{1/2}(\p\GF_R)} =  \| f_\Gr \|_{H^{1/2}(\Rbb^2)} \sim 1.
$
So, it suffices to show
\beq\label{GFbounded}
\lim_{R \to \infty} \| (\Gl I- \Kcal_{\p\GF_R})[\Gvf_{\Gr}] \|_{H^{1/2}(\p\GF_R)} = 0.
\eeq

For $c>0$   we set
$$
U_R^c := \{ (x,x_3)  :  x \in U_R, \ \mbox{dist}(x, \p U_R) \ge c \}.
$$
Let $C$ be a constant such that $\GG^s  \subset \Rbb^3 \setminus U_R^C$ where
$\GG^s$ denotes  the projection of $\GS_R^s$ onto the $x$-plane.  Note that we can choose such a constant independently of $R>1$.
Let $\Gc_1(X) = \Gc_1(x,x_3) =\Gc_1(x)$ be a smooth function supported in $U_R^C $ such that $\Gc_1=1$ on $U_R^{2C}$, and let $\Gc_2:=1-\Gc_1$. Then we have
$$
\| (\Gl I- \Kcal_{\p\GF_R})[\Gvf_\Gr] \|_{H^{1/2}(\p\GF_R)} \le \sum_{j=1}^2 \| \Gc_j(\Gl I- \Kcal_{\p\GF_R})[\Gvf_\Gr] \|_{H^{1/2}(\p\GF_R)} .
$$

Since
$
\Gc_1(\Gl I- \Kcal_{\p\GF_R})[\Gvf_\Gr] (x,x_3) = \Gc_1( |\Gl| f_\Gr -\frac{1}{2} (P_2 * f_\Gr) ),
$
it follows from (ii) of Lemma \ref{lem:limitGF}  that
\beq
\| \Gc_1(\Gl I- \Kcal_{\p\GF_R})[\Gvf_\Gr] \|_{H^{1/2}(\p\GF_R)} \to 0 \ \text{ as } \  R \to \infty.
\eeq

To estimate $\| \Gc_2(\Gl I- \Kcal_{\p\GF_R})[\Gvf_\Gr] \|_{H^{1/2}(\p\GF_R)}$, let
$$
\GG:= \p \GF_R \cap (\Rbb^3 \setminus U_R^{2C}).
$$
Then we have
$$
\| \Gc_2(\Gl I- \Kcal_{\p\GF_R})[\Gvf_\Gr] \|_{H^{1/2}(\p\GF_R)} = \| \Gc_2 \Kcal_{\p\GF_R}[\Gvf_\Gr] \|_{H^{1/2}(\GG)}.
$$

Note that the shape of $\GG$ is independent of $R$.
We use the following characterization of the space $H^{1/2}(\GG)$ (see, e.g., \cite{GT}):
\beq\label{def_H1/2}
\| h \|_{H^{1/2}(\GG)}^2= \| h \|_{L^2(\GG)}^2 + \int_{\GG} \int_{\GG} \frac{|h(X)-h(Z)|^2}{|X-Z|^3} dS(X) dS(Z).
\eeq

Let $K(X,Y)$ be the integral kernel of $\Kcal_{\p\GF_R}$, namely,
$$
K_R(X,Y) = \frac{1}{4\pi} \frac{\la Y-X, \nu(Y) \ra}{|X-Y|^3}.
$$
If $X, Z \in \mbox{supp}(\Gc_2)$ and $Y \in \mbox{supp}(\Gvf_{\Gr})$, then
$$
|X-Z| \lesssim 1, \quad |X-Y| \gtrsim R, \quad |Z-Y| \gtrsim R.
$$
Thus, $|K(X,Y)| \lesssim R^{-2}$. It then follows from \eqnref{fRx} that
\beq\label{100}
\| \Gc_2 \Kcal_{\p\GO_R}[\Gvf_\Gr] \|_{L^2(\GG)} \lesssim R^{-2} \int_{\Rbb^2} |f_\Gr(y)|dy \lesssim R^{-1-\Gs} \int_{\Rbb^2} |(\chi\Gz)(y)|dy \lesssim R^{-1-\Gs}.
\eeq
We also have
\begin{align*}
&|\Gc_2(X) K(X,Y) - \Gc_2(Z) K(Z,Y)| \\
& \le |\Gc_2(X)- \Gc_2(Z)| |K(X,Y)| + |\Gc_2(Z)| |K(X,Y)-K(Z,Y)| \\
& \lesssim R^{-2} |X- Z| + R^{-3} |X- Z| \le R^{-2} |X- Z|.
\end{align*}
Thus,
\begin{align*}
|\Gc_2(X) \Kcal_{\p\GF_R}[\Gvf_\Gr](X) - \Gc_2(Z) \Kcal_{\p\GF_R}[\Gvf_\Gr](Z)|
&\lesssim R^{-2}|X- Z| \int_{\Rbb^2} |f_\Gr(y)|dy \\
&\le R^{-1-\Gs} |X- Z| \int_{\Rbb^2} |(\chi\Gz)(y)|dy \\
& \lesssim R^{-1-\Gs} |X- Z|.
\end{align*}
It then follows that
$$
\int_{\GG} \int_{\GG} \frac{|\Gc_2(X) \Kcal_{\p\GF_R}[\Gvf_{\Gr}](X) - \Gc_2(Z) \Kcal_{\p\GF_R}[\Gvf_{\Gr}](Z)|^2}{|X- Z|^3} dS(X) dS(Z) \lesssim R^{-1-\Gs},
$$
which together with \eqnref{100} implies \eqnref{GFbounded}.
This completes the proof.
\end{proof}

\section*{Acknowledgments}
We thank Graeme Milton for useful discussion on negative NP eigenvalues.


\end{document}